\documentclass{amsart}

\usepackage{mathtools}
\usepackage{xcolor}
\usepackage{amsmath,amssymb}
\usepackage{cite}

\newtheorem{theorem}{Theorem}
\newtheorem{lemma}[theorem]{Lemma}
\newtheorem{corollary}[theorem]{Corollary}
\newtheorem{proposition}[theorem]{Proposition}
\theoremstyle{remark}

\newtheorem*{remark}{Remark}
\newtheorem*{example}{Example}

\title{Relatively free algebras of finite rank}
\author{Thiago Castilho de Mello}
\address{Universidade Federal de S\~ao Paulo, Instituto de Ci\^encia e Tecnologia}
\email{tcmello@unifesp.br}

\author{Felipe Yukihide Yasumura}
\address{Universidade de S\~ao Paulo, Instituto de Matem\'atica e Estat\'istica}
\email{fyyasumura@ime.usp.br}

\dedicatory{Dedicated to Professor Antonio Giambruno on his 70th birthday.}

\begin{document}
\maketitle

\begin{abstract}
  Let $\mathbb{K}$ be a field of characteristic zero and $B=B_0+B_1$ a finite dimensional associative superalgebra. In this paper we investigate the polynomial identities of the relatively free algebras of finite rank of the variety $\mathfrak V$ defined  by the Grassmann envelope of $B$. We also consider the $k$-th Grassmann Envelope of $B$, $G^{(k)}(B)$, constructed with the $k$-generated Grassmann algebra, instead of the infinite dimensional Grassmann algebra. We specialize our studies for the algebra $UT_2(G)$ and $UT_2(G^{(k)})$, which can be seen as the Grassmann envelope and $k$-th Grassmann envelope, respectively, of the superalgebra $UT_2(\mathbb{K}[u])$, where $u^2=1$.
\end{abstract}

\section{Introduction}
\label{Introduction}

In this paper $\mathbb{K}$ will denote a field of characteristic 0. If $X=\{x_1,x_2,\dots\}$ is an infinite countable set, we denote by $\mathbb{K}\langle X\rangle$ the free associative unitary algebra freely generated by $X$. If $A$ is an associative algebra, we say that it is an \emph{algebra with polynomial identity} (PI-algebra, for short) if there exists a nonzero polynomial $f=f(x_1,\dots,x_n)\in\mathbb{K}\langle X\rangle$ such that $f(a_1,\dots, a_n)=0$, for arbitrary $a_1$, \dots, $a_n\in A$. In this case, we say that $f$ is a (polynomial) identity of $A$, or simply that $A$ satisfies $f$.

If $A$ is a PI-algebra the set $T(A)=\{f\in\mathbb{K}\langle X\rangle\,|\, f \text{ is an identity of } A\}$ is an ideal of $\mathbb{K}\langle X\rangle$  invariant under any endomorphism of the algebra $\mathbb{K}\langle X\rangle$. An ideal with this property is called a \emph{T-ideal} or \emph{verbal ideal} of $\mathbb{K}\langle X\rangle$. We refer the reader to \cite{drenskybook,GZ} for the basic theory of PI-algebras.

The T-ideals play a central role in the theory of PI-algebras, and they are often studied through the equivalent notion of varieties of algebras. If
$\mathcal{F}$ is a subset of $\mathbb{K}\langle X\rangle$, the class of all algebras satisfying the identities from $\mathcal{F}$ is called \emph{the variety of (associative) algebras defined by $\mathcal{F}$} and denoted by $\mathrm{var}(\mathcal{F})$.
Given $\mathfrak{V}$ and $\mathfrak{W}$ varieties of algebras, we say that $\mathfrak{W}$ is a subvariety of $\mathfrak{V}$ if $\mathfrak{W}\subseteq \mathfrak{V}$. If $\mathfrak{V}$ is a variety of algebras, we denote by $T(\mathfrak{V})$ the set of identities satisfied by all algebras in $\mathfrak{V}$.
One can easily see that $T(\mathfrak{V})$ is a T-ideal of $\mathbb{K}\langle X\rangle$. It is called \emph{the T-ideal of $\mathfrak{V}$}.
If $\mathfrak{V}=\mathrm{var}(\mathcal{F})$, we say that the elements of $T(\mathfrak{V})$ are \emph{consequences of} (or \emph{follow from}) the elements of $\mathcal{F}$.

Let $\mathfrak V$ be a variety of associative algebras. In the theory of algebras with polynomial identities, an important role is played by the so called \textit{relatively free algebras of $\mathfrak V$}. The relatively free algebra of $\mathfrak V$ freely generated by a set $X$ is an algebra $F_X(\mathfrak{V}) \in \mathfrak V$, with an inclusion map $\iota : X \xhookrightarrow{}{}$ $F_X(\mathfrak{V})$, satisfying the following universal property:

Given any algebra $A\in \mathfrak V$, and a map $\varphi_0: X \longrightarrow A$, there exists a unique algebra homomorphism $\varphi:F_X(\mathfrak V) \longrightarrow A$ such that $ \varphi\circ \iota =\varphi_0$. 

It is a simple exercise to show that \[F_X(\mathfrak V)\cong \dfrac{\mathbb{K}\langle X \rangle}{T(\mathfrak{V})}.\] Moreover, it is well known that for two given sets $X$ and $Y$, the algebras $F_X(\mathfrak V)$ and  $F_Y(\mathfrak V)$ are isomorphic if and only if $X$ and $Y$ have the same cardinality. Therefore if $|X|=k\in \mathbb N$, we denote  $F_X(\mathfrak V)$ simply by  $F_k(\mathfrak V)$ and if $X$ is a countably infinite set we denote it simply by  $F(\mathfrak V)$.

The first studies about relatively free algebras are due to Procesi (see \cite{ProcesiBook, Procesi}), when dealing with the so called \emph{algebra of generic matrices}, which is isomorphic to the relatively free algebra in the variety generated by $M_n(\mathbb{K})$. This algebra is a fundamental object in invariant theory and has noteworthy properties. For instance, it has no zero divisors and one can work with its quotient ring, the so called \emph{generic division algebra} (see \cite{Formanek}). Another interesting property is that given a polynomial $f(x_1,\dots,x_k)$, we have that $f$ is a central polynomial for $M_n(\mathbb{K})$ if and only if $\overline{f(x_1,\dots,x_k)}$ is a central element of $F_k(M_n(\mathbb{K}))$. Such properties do not hold in any variety. A simple example can be seen in the variety generated by the algebra $M_{1,1}=\begin{pmatrix} G_0 & G_1 \\ G_1 & G_0 \end{pmatrix}$, where $f(x,y)=[x,y]^2$ is a central element in $F_2(M_{1,1})$, but it is not a central polynomial for $M_{1,1}$ (see \cite{KoshM, Popov}). 

We say that a variety of algebras $\mathfrak{V}$ has a finite basic rank if $\mathfrak{V}=\mathrm{var}(A)$, where $A$ is a finitely generated
algebra. The minimal number of generators of such an algebra is the {\it{basic rank} of the variety $\mathfrak{V}$}.

Of course that the variety generated by the algebra $F(\mathfrak V)$ is $\mathfrak V$ itself, and that for any $k$, $F_k(\mathfrak{V})\in \mathfrak V$, but it is not true in general that there exists $k$ such that $F_k(\mathfrak{V})$ generates $\mathfrak{V}$. 

The basic rank of a variety $\mathfrak {V}$ can be characterized in terms of its relatively free algebras, as we can see in the following easy-to-prove proposition.

\begin{proposition}
	The basic rank of a variety $\mathfrak{V}$ is the least integer $k$ such that $\mathfrak{V}=\mathrm{var}(F_k(\mathfrak{V}))$
\end{proposition}

As examples, we mention that the variety generated by the Grassmann algebra of an infinite dimensional vector space has infinite basic rank, while the algebra of $n\times n$ matrices over the field ($n>1$) generates a basic rank 2 variety (since $M_n(\mathbb{K})$ is a 2-generated algebra).

A natural problem in the theory of PI-algebras is to classify in some sense the subvarieties in a given variety of algebras $\mathfrak V$. A very important role in this direction is played by the exponent of a variety $\mathfrak V$. Proving a conjecture of Amitsur, Giambruno and Zaicev showed that for any variety of associative algebras over a field of characteristic zero the exponent exists and is an integer \cite{GiamZa1,GiamZa2} (see also \cite{GZ}). Therefore, it was natural to classify varieties of algebras in terms of its exponents. A successful approach was the classification in terms of forbidden algebras. For example, Kemer showed that the varieties of exponent 1 are exactly those varieties not containing the infinite dimensional Grassmann algebra and the $2\times 2$ algebra of upper triangular matrices. Similar results were given for varieties of exponent 2, with a list of 5 forbidden algebras (see \cite{GiamZa3}).

The classification of subvarieties of important varieties of algebras were also studied. For instance, the classification of the subvarieties of the variety $\mathfrak{G}$, generated by the infinite dimensional Grassmann algebra, was given by La Mattina \cite{LaM}, and subvarieties of the variety generated by $M_2(\mathbb{K})$ were studied by Drensky \cite{Drenskysubv}. The classification of subvarieties of a given variety is a difficult task. A less accurate, but interesting approach, is the classification up to asymptotic equivalence of varieties. This notion was introduced by Kemer in \cite{Kemersubv}, where he classified subvarieties of a variety satisfying the identity $St_4$ (the standard identity of degree 4), up to asymptotic equivalence. We say that two T-ideals are \emph{asymptotically equivalent} if they satisfy the same proper polynomials from a certain degree on. We recall that proper polynomials are those which are linear combinations of products of commutators. In a similar way, two varieties are asymptotically equivalent if their T-ideals are so. This notion was also used to classify, up to asymptotic equivalence, the subvarieties of the variety $\mathfrak M_5$, generated by the identity $[x_1,x_2][x_3,x_4,x_5]$ (see \cite{DDN}).  Such variety can be realized as the variety generated by the algebra $A=\begin{pmatrix}G_0 & G \\ 0 & G \end{pmatrix}$. This is one of the forbidden algebras in the classification of varieties of exponent 2.

If $\mathfrak V$ is a variety of associative algebras of infinite basic rank, we have a lattice of T-ideals
\[T(F_1(\mathfrak V))\supseteq T(F_2(\mathfrak V))\supseteq \cdots T(F_k(\mathfrak V))\supseteq T(F_{k+1}(\mathfrak V))\supseteq  \cdots \supseteq T(F(\mathfrak V))=T(\mathfrak V).\]
As a consequence of Lemma \ref{lem} below, one can easily see that there is an infinite number of proper inclusions above.

A natural but difficult problem in general is to describe for all $k$, the T-ideals $T(F_k(\mathfrak V))$. This task was realized only for a small list of varieties of infinite basic rank, namely: the variety $\mathfrak{G}$, the variety $\mathfrak M_5$ \cite{DM}, and the variety generated by the algebra $M_{1,1}(G)$. The last only for $k=2$ (see \cite{KMb}).

In the mentioned examples, the knowledge of the the identities of the relatively free algebras of finite rank, were useful to give an alternative description to the subvarieties of the given variety. For instance, if $A\in \mathfrak G$ is a unitary algebra, it is PI-equivalent to $\mathbb{K}$, $G$, or $F_{2k}(\mathfrak G)$, for some $k$, and if $A\in \mathfrak{M}_5$, then, it is PI-equivalent to $\mathbb{K}$, $UT_2(\mathbb{K})$, $E$, $F_{2k}(\mathfrak M_5)$ or $F_{2k}(\mathfrak M_5)\oplus G$, for some $k$.

We believe that the knowledge of the identities of the relatively free algebras of finite rank of a given variety of infinite basic rank may play an important role in the description of its subvarieties. This is a motivation to the study of such identities.

From Kemer's theory \cite{kemer} we know that every finitely generated algebra satisfies the same identities of a finite dimensional algebra. In light of this, given a variety $\mathfrak V$, it is interesting to find finite dimensional algebras $A_k\in \mathfrak V$ such that $T(A_k)=T(F_k(\mathfrak V))$, for all $k$. This was done to the above-mentioned examples. In those cases, it was verified that for all $k$, the algebra $A_k$ was obtained with the construction we describe below. 

It is well known from Kemer's theory that the variety $\mathfrak V$ is generated by the \emph{Grassmann envelope} of a suitable finite dimensional superalgebra $B=B_0+B_1$. Recall that the Grassmann  envelope of $B$ is given by $G(B)=G_0\otimes B_0+G_1\otimes B_1$, i.e., the even part of the superalgebra $G\otimes B$.
Similarly, one can define the $k$-th Grassmann envelope of $B$ as $G^{(k)}(B) = G^{(k)}_0\otimes B_0+G^{(k)}_1 \otimes B_1$, where $G^{(k)}$ is the Grassmann algebra of a $k$-dimensional vector space over $K$.

If $B=B_0+B_1$ is the superalgebra (which exists by Kemer theory) satisfying $T(G(B))=T(\mathfrak V)$, the above-mentioned examples satisfy the following interesting property: 

\begin{equation}\label{=}
T(F_k(\mathfrak {V}))=T(G^{(k)}(B)),
\end{equation}
for any $k$, in the case $\mathfrak V=\mathfrak G$ or $\mathfrak V=\mathfrak M_5$. In the case $\mathfrak V=var (M_{1,1})$ we only know it for $k=2$ (see \cite{DM}).

In light of the these results, it is an interesting problem to compare the T-ideals of $T(G^{(k)}(B))$ and $T(F_k(G(B)))$ for a given finite dimensional superalgebra $B$.

In the present paper, we obtain partial results on this problem for the variety $\mathfrak G_2$, generated by $UT_2(G)$. We show that the equality (\ref{=}) does not hold for this variety. 

We divide this paper as follows. We construct different models for the relatively free algebras in Section \ref{Models}, which will give different approaches for the problem. In Section \ref{genrem} we prove general facts that hold for the relatively free algebras of finite rank. In Section \ref{sec4}, we investigate the polynomial identities of $UT_2(G^{(k)})$, and exhibit a basis of identities when $2\le k\le 5$. Finally, in Section \ref{sec5}, we investigate the polynomial identities of $F_k(UT_2(G))$.

\section{Models for relatively free algebras}
\label{Models}

The relatively free algebras are quotients of the polynomial algebra $\mathbb{K}\langle X \rangle$. In particular, its elements are cosets of noncommutative polynomials. In order to have a more concrete object to work with, we will present some models of these relatively free algebras, which can simplify the problem of working with quotient classes.

The most simple example of a model for a relatively free  algebra is the \emph{algebra of generic matrices} (for the variety generated by an $n\times n$ matrix algebra over an infinite field $\mathbb{K}$). By a \emph{model}, we mean an algebra isomorphic to the given relatively free algebra.

Let $n$ be a positive integer, $X=\{x_{ij}^{(k)}\,|\, i,j\in \{1,\dots, n\}, k\in \mathbb{N}\}$ and $\mathbb{K}[X]$ be the algebra of commutative polynomials on the variables of $X$.
The algebra of generic $n\times n$ matrices is the subalgebra of $M_n(\mathbb{K}[X])$ generated by the matrices 
\[\xi_k=
\begin{pmatrix}
x_{11}^{(k)} & x_{12}^{(k)} & \cdots & x_{1n}^{(k)}\\
x_{21}^{(k)} & x_{22}^{(k)} & \cdots & x_{2n}^{(k)}\\
\vdots & \vdots & \ddots & \vdots \\
x_{n1}^{(k)} & x_{n2}^{(k)} & \cdots & x_{nn}^{(k)}\\
\end{pmatrix},\quad \text{for } k\in\mathbb{N}.\]

In a similar way, one can construct a model for a relatively free algebra of a variety generated by a finite dimensional algebra $A$. For, one only needs to fix a basis $B=\{v_1,\dots,v_n\}$ of $A$, and consider a subalgebra of $\mathbb{K}[X]\otimes A$, (where $X=\{x_i^{(k)}\,|\, i\in \{1,\dots, n\}, k\in \mathbb N\}$) generated by the elements 
\[\xi_k=\sum_{i=1}^{n}x_i^{(k)}\otimes v_i, \quad \text{for } k\in\mathbb{N}. \]

On the other hand, when dealing with a variety of infinite basic rank, the above construction is not possible.

Examples of models for relatively free algebras of infinite basic rank varieties were given by Berele in \cite{Berele}. More specifically, Berele constructed models for the relatively free algebras of varieties generated by $M_{n}(G)$ and $M_{a,b}(G)$ (the so called T-prime varieties), as algebras of matrices over the free supercommutative algebra $\mathbb{K}[X;Y]$.

We say that a superalgebra $A=A_0+A_1$ is supercommutative if its Grassmann envelope is commutative, i.e., if for any $a,b\in A_0\cup A_1$, one has $ab=(-1)^{\deg a \deg b}ba$.

Considering $\mathbb{K}$ as an infinite field, we proceed with a construction of a free supercommutative superalgebra.

Let $X$ and $Y$  be countably infinite sets. We build the algebra $\mathbb{K}\langle X\cup Y\rangle$ and induce on it a $\mathbb Z_2$-grading by defining $\deg x=0$, $x\in X$, and $\deg y=1$,
$y\in Y$. The algebra $\mathbb{K}\langle X\cup Y\rangle$ with such grading is called the \emph{free $\mathbb{Z}_2$-graded algebra}. If  $I$ is the ideal generated by the elements
$ab-(-1)^{\deg a \deg b}ba$, $a,b\in X\cup Y$, we define the \emph{ free supercommutative algebra}, denoted by $\mathbb{K}[X;Y]$,  as the quotient algebra
\[\mathbb{K}[X;Y]=\frac{\mathbb{K}\langle X\cup Y\rangle}{I}. \]

One can easily verify that given any supercommutative superalgebra $A=A_0+A_1$, and a map $\varphi_0:X\cup Y\longrightarrow A$ such that $\varphi_0(x)\in A_0$ if $x\in X$ and $\varphi_0(y)\in A_1$ if $y\in Y$, there exists a unique homomorphism $\varphi:\mathbb{K}[X;Y]\longrightarrow A$ which extends $\varphi_0$. 

Now given a finite dimensional superalgebra $B=B_0+B_1$, we proceed with a construction of a model for the relatively free algebra of the variety generated by $G(B)$. 
We remark that such is a completely general construction, since any variety of associative algebras is generated by $G(B)$, for some $B$, although given an arbitrary variety it is not a simple task to determine one such $B$.

Let us fix $\{u_1,\dots, u_r\}$ a basis of $B_0$ and $\{v_1,\dots, v_s\}$ a basis of $B_1$ and let us consider the sets $X=\{x_j^{(i)}\,|\, i\in \mathbb{N}, j\in \{1,\dots, r\}\}$ and $Y=\{y_j^{(i)}\,|\, i\in \mathbb{N}, j\in \{1,\dots, s\}\}$.

We consider the free supercommutative algebra $\mathbb{K}[X;Y]$ and for each $i\in \mathbb{N}$, we define $\xi_i \in B\otimes \mathbb{K}[X;Y]$ as

\[\xi_i=\sum_{j=1}^r u_j\otimes x_j^{(i)}+\sum_{j=1}^s v_j\otimes y_j^{(i)}\]

Then we have:

\begin{proposition}\label{supermodel}
	Let $n\in \mathbb{N}$, and define $\mathbb{K}[\xi_1,\xi_2,\dots]$ and $\mathbb{K}[\xi_1\,\dots,\xi_n]$ as the  subalgebras of $B\otimes \mathbb{K}[X;Y]$ generated by the elements $\xi_1,\xi_2,\dots$ and by the elements $\xi_1,\dots,\xi_n$, respectively. Then, the following isomorphisms hold:
	\[\mathbb{K}[\xi_1,\xi_2,\dots]\cong F(G(B))\] \[\mathbb{K}[\xi_1,\dots,\xi_n]\cong F_n(G(B))\]
\end{proposition}
\begin{proof}
	Define the algebra homomorphism $\eta:\mathbb{K}\langle t_1,t_2\dots \rangle \longrightarrow\mathbb{K}[\xi_1,\xi_2,\dots]$ by $\eta(t_i)=\xi_i$.  In particular, if $f(t_1,\dots,t_k)\in \mathbb{K}\langle t_1,t_2\dots \rangle$, then $\eta(f)=f(\xi_1,\dots,\xi_k)$.
	
	Of course $\eta$ is surjective. Once we show that $\ker \eta=T(G(B))$, the result is proved.
	
	Suppose $f\in \ker \eta$. This means that $f(\xi_1,\dots,\xi_k)=0$. 
	
	For each $i\in \{1,\dots,k\}$ we consider arbitrary elements $a_i$ of $G(B)$. These can be written as 
	\[a_i=\sum_{j=1}^r u_j\otimes g_j^{(i)}+\sum_{j=1}^s v_j\otimes h_j^{(i)}\] where $g_j^{(i)}$ and  $h_j^{(i)}$ are arbitrary even and odd elements of the Grassmann algebra respectively. Since $\mathbb{K}[X;Y]$ is the free supercommutative algebra, there exists a
	homomorphism $\varphi:\mathbb{K}[X;Y]\longrightarrow G$ extending the map $\varphi_0:X\cup Y\longrightarrow G$, given by $\varphi_0(x_j^{(i)})=g_j^{(i)}$ and $\varphi_0(y_j^{(i)})=h_j^{(i)}$. 
	
	From this, we define the homomorphism of algebras $\phi:B\otimes \mathbb{K}[X;Y]\longrightarrow B\otimes G$, given by $\varphi$ in $\mathbb{K}[X;Y]$ and fixing $B$. Then, for each $i$,
	\[a_i=\phi\left(\sum_{j=1}^r u_j\otimes x_j^{(i)}+\sum_{j=1}^s v_j\otimes y_j^{(i)}\right)=\phi(\xi_i)\]
	As a consequence, \[f(a_1,\dots,a_k)=\phi(\eta(f))=0,\] which means $f\in T(G(B))$.
	
	Conversely, suppose  $f\in T(G(B))$. We will show that $f(\xi_1,\dots,\xi_k)=0$.
	
	Write \[f(\xi_1,\dots,\xi_k)=\sum_{j=1}^r u_j\otimes m_j+\sum_{j=1}^s v_j\otimes n_j\] where $m_j$ and $n_j$ are $\mathbb{Z}_2$-graded polynomials of even and odd degree respectively in the commutative and anticommutative variables $x_p^{(q)}$ and $y_p^{(q)}$ of $\mathbb{K}[X;Y]$.
	
	As we have already shown, if \[a_i=\sum_{j=1}^r u_j\otimes g_j^{(i)}+\sum_{j=1}^s v_j\otimes h_j^{(i)}\] are arbitrary elements of $G(B)$, we have $f(a_1,\dots,a_k)=\phi(f(\xi_1,\dots,\xi_k)=\sum_{j}u_j\otimes m_j(g;h)+\sum_j v_j\otimes n_j(g,h)$. Since the  $a_i$s are arbitrary, so are the homogeneous elements $g$ and $h$ of even and  odd homogeneous degree in $G$. As a consequence, since $f\in T(G(B))$, $m_j$ and $n_j$ are $\mathbb{Z}_2$-graded identities of $G$ and since $\mathbb{K}[X;Y]$ is the free supercomutative algebra, it follows that $m_j=n_j=0$ in $\mathbb{K}[X;Y]$. What means that $f(\xi_1,\dots,\xi_k)=0$, finishing the proof.
	
	The case of $F_n(G(B))$ is analogous.
\end{proof}

It should be remarked that such model has appeared in \cite[Section 3.8]{GZ}.

Another possible model for relatively free algebras of some special kind of varieties is presented now.

Consider $R$ a PI-algebra and let $A$ be a subalgebra of $M_n(R)$ generated by matrix units $e_{ij}$. We give a model for the relatively free algebra of the variety generated by $A$ as a subalgebra of matrices over the relatively free algebra of $R$. The general construction can be find in the paper \cite[Lemma 6]{CdM}. Here we present it, as an example, for the particular case of $A=UT_2(G)$, which we will use below in the paper.

\begin{example}
	Let $\mathcal{U}$ and $\mathcal{U}_{k}$ be the subalgebras of $UT_2(F(\mathfrak G))$ generated by the generic matrices $\xi_1,\xi_2,\dots$ and $\xi_1,\dots,\xi_k$, respectively, where
	
	\[\xi_i=\begin{pmatrix}
	x_{11}^{(1)}+T(G) & x_{12}^{(i)}+T(G)\\
	0                 & x_{22}^{(2)}+T(G)
	\end{pmatrix}.\]
	
	Then 
	\begin{eqnarray*}
	&\mathcal {U} \cong F(UT_2(G))),&\\%
	&\mathcal {U}_k \cong F_k(UT_2(G))), \text{ for } k\in \mathbb{N}.&
	\end{eqnarray*}
	
	It is interesting to observe that when dealing with this model, one transfers the problem of dealing with cosets of matrices, to dealing with cosets of elements in the entries of such matrices. When the relatively free algebra of $R$ is well known, this construction is very useful. For instance, if $R$ is the field, its relatively free algebra is the polynomial algebra in commuting variables, and we are in the classical case of generic matrices. 
	In this paper we will deal with this model when $R$ is the Grassmann algebra, but since its relatively free algebra is easy to handle with, this will help us to obtain our results.
\end{example}

\section{General remarks\label{genrem}}
Let $\mathcal{A}$ be a finite-dimensional associative superalgebra.

As mentioned in \cite{DM}, if $k_1\le k_2$, then
\begin{equation}\label{eq}
T(\mathcal{A})\subseteq T(F_{k_2}(\mathcal{A}))\subseteq T(F_{k_1}(\mathcal{A})).
\end{equation}
Moreover,{\cite[Lemma 8]{DM}} the authors prove:

\begin{lemma}\label{lem}
	$T(F_n(\mathcal{A}))\cap\mathbb{K}\langle x_1,\ldots,x_n\rangle=T(\mathcal{A})\cap\mathbb{K}\langle x_1,\ldots,x_n\rangle$.\qed
\end{lemma}

As a consequence, we have the following:
\begin{proposition}
	$T(\mathcal{A})=\bigcap_{n\ge1}T(F_n(\mathcal{A}))$. In particular, $F(\mathcal{A})$ is a subdirect product of the $\{F(F_n(\mathcal{A}))\}_{n\in\mathbb{N}}$.
\end{proposition}
\begin{proof}
	Clearly $T(\mathcal{A})\subseteq\bigcap_{n\ge1}T(F_n(\mathcal{A}))$. Conversely, given
	$$
	f=f(x_1,\ldots,x_n)\in\bigcap_{n\ge1}T(F_n(\mathcal{A})),
	$$
	by Lemma \ref{lem}, we have $f\in T(F_n(\mathcal{A}))\cap\mathbb{K}\langle x_1,\ldots,x_n\rangle\subseteq T(\mathcal{A})$.
\end{proof}

Finally, we have the following alternative description of $F(\mathcal{A})$. We start with a lemma:
\begin{lemma}
	If $i\le j$, then there exists an algebra monomorphism $u_{ij}:F_i(\mathcal{A})\to F_j(\mathcal{A})$. Moreover, if $i\le j\le k$, then $u_{ik}=u_{jk}u_{ij}$.
\end{lemma}
\begin{proof}
	Since $F_i(\mathcal{A})$ is free in $\mathrm{var}(\mathcal{A})$, a homomorphism from $F_i(\mathcal{A})$ to an algebra in this variety is defined by a choice of images of the free generators of $F_i(\mathcal{A})$. So we can let $u_{ij}$ send the free generators $\xi_1,\ldots,\xi_i$ of $F_i(\mathcal{A})$ to the first $i$ free generators of $F_j(\mathcal{A})$. If the image of some element is zero in $F_j(\mathcal{A})$, then it is a polynomial identity of $F_i(\mathcal{A})$, so it will be zero in $F_i(\mathcal{A})$. Thus, this map is injective. The last assertion is immediate from the construction of the $u_{ij}$.
\end{proof}

The last lemma says that the pair $\left((\mathcal{A}_i)_{i\in\mathbb{N}},(u_{ij})_{i\le j}\right)$ is a direct system.

\begin{proposition}
	$F(\mathcal{A})=\lim\limits_{\longrightarrow} F_i(\mathcal{A})$.
\end{proposition}
\begin{proof}
	For each $i$, let $u_i:F_i(\mathcal{A})\to F(\mathcal{A})$ be the map sending the free generators of $F_i(\mathcal{A})$ to the first $i$ free generators of $F(\mathcal{A})$. Clearly $u_ju_{ij}=u_i$, for all $i\le j$.
	
	Now, consider a target $(\mathcal{B},(\phi_i)_{i\in\mathbb{N}})$, that is, an algebra $\mathcal{B}$ together with homomorphisms $\phi_i:F_i(\mathcal{A})\to\mathcal{B}$ such that $\phi_ju_{ij}=\phi_i$. We define $u:F(\mathcal{A})\to\mathcal{B}$ in the free generators $\xi_i$ via $u(\xi_i):=\phi_i(\xi_i)$ (the same image of $\phi_i$ applied in the last generator of $F_i(\mathcal{A})$). So clearly $uu_i=\phi_i$, for each $i\in\mathbb{N}$. This proves that $\lim\limits_{\longrightarrow} F_i(\mathcal{A})=F(\mathcal{A})$.
\end{proof}

\begin{corollary}
	For any $\mathcal{B}\in\mathrm{var}(\mathcal{A})$, one has
	$$
	\mathrm{Hom}(F(\mathcal{A}),\mathcal{B})=\lim\limits_{\longleftarrow}\mathrm{Hom}(F_i(\mathcal{A}),\mathcal{B}).
	$$\qed
\end{corollary}
The former corollary has an intuitive (and somewhat obvious) interpretation. Let $f\in F(\mathcal{A})$. It is known that the following three assertions are equivalent:
\begin{enumerate}
	\renewcommand{\labelenumi}{(\roman{enumi})}
	\item $f$ is a polynomial identity of $\mathcal{B}$,
	\item $f\in\mathrm{Ker}\,\Psi$, for all $\Psi\in\mathrm{Hom}(F(\mathcal{A}),\mathcal{B})$.
	\item $f\in\mathrm{Ker}\,\Psi$, for all $\Psi\in\mathrm{Hom}(F_j(\mathcal{A}),\mathcal{B})$, for a sufficiently large $j$ (indeed, it is enough to take a $j$ greater or equal to the number of variables of $f$).
\end{enumerate}
The last corollary states the equivalence between (ii) and (iii).

\section{Polynomial identities for $UT_2(G^{(k)})$\label{sec4}}
In this section, we investigate the polynomial identities of $UT_2(G^{(k)})$. We find a explicit set of polynomials that, together with some class of polynomials, generate the T-ideal of polynomial identities of $UT_2(G^{(k)})$.

One may notice that if $A$ and $B$ are algebras such that $R=\begin{pmatrix} A & M \\ 0 & B \end{pmatrix}$ is an algebra, then $T(A)T(B)\subseteq T(R)$. Verifying if the above inclusion is an equality is a more difficult task. In some cases, the approach of Lewin's Theorem applies (see \cite{Lewin} or \cite[Corollary 1.8.2]{GZ} for a more suitable version). Some results in this direction are given in the paper \cite{CdM}, where the authors describe conditions under which the T-ideal of a block-triangular matrix algebra over an algebra $A$ factors as the product of the ideals of the blocks. But one can see that the algebras $G^{(k)}$ do not satisfy the necessary hypothesis to that, namely the existence of a partially multiplicative basis for its relatively free algebras, so in this paper we try a different approach.

We let $\mathbb{K}$ be a field of characteristic zero.
\begin{lemma}\label{lem1}
	For any $t\in\mathbb{N}$, the following are consequences of $[x_1,x_2,x_3][x_4,x_5,x_6]=0$:
	\begin{enumerate}
		\item $[y_1,y_2,y_3]p(x_1,\ldots,x_t)[z_1,z_2,z_3]=0$,
		\item $([y_1,y_2][y_3,y_4]+[y_1,y_3][y_2,y_4])p(x_1,\ldots,x_t)[z_1,z_2,z_3]=0$,
		\item $[y_1,y_2,y_3]p(x_1,\ldots,x_t)([z_1,z_2][z_3,z_4]+[z_1,z_3][z_2,z_4])=0$,
		\item $([y_1,y_2][y_3,y_4]+[y_1,y_3][y_2,y_4])p(x_1,\ldots,x_t)([z_1,z_2][z_3,z_4]+[z_1,z_3][z_2,z_4])=0$.
	\end{enumerate}
	where $p(x_1,\ldots,x_t)$ is any multilinear polynomial.
\end{lemma}
\begin{proof}
	The first one follows from
	\begin{align*}
	[y_1,y_2,y_3]p(x_1,\ldots,x_t)[z_1,z_2,z_3]=&[[y_1,y_2,y_3],p(x_1,\ldots,x_t)][z_1,z_2,z_3]\\%
	&+p(x_1,\ldots,x_t)[y_1,y_2,y_3][z_1,z_2,z_3].
	\end{align*}
	Working modulo the identity $[x_1,x_2,x_3][x_4,x_5,x_6]$, for the second one we have
	\begin{align*}
	0&=[y_1,y_2^2,y_3]p(x_1,\ldots,x_t)[z_1,z_2,z_3]=([[y_1,y_2]y_2,y_3]+[y_2[y_1,y_2],y_3])p(x_1,\ldots,x_t)[z_1,z_2,z_3]\\%
	&=([y_1,y_2][y_2,y_3]+[y_1,y_2,y_3]y_2+[y_2,y_3][y_1,y_2]+y_2[y_1,y_2,y_3])p(x_1,\ldots,x_t)[z_1,z_2,z_3]\\%
	&=([y_1,y_2][y_2,y_3]+[y_1,y_2][y_2,y_3]+[[y_2,y_3],[y_1,y_2]])p(x_1,\ldots,x_t)[z_1,z_2,z_3]\\%
	&=2[y_1,y_2][y_2,y_3]p(x_1,\ldots,x_t)[z_1,z_2,z_3].
	\end{align*}
	Linearizing the above identity, we obtain (2). Analogously we obtain (3) and (4).
\end{proof}

We fix $m\in\mathbb{N}$.
\begin{lemma}\label{lem2}
	The polynomials
	\begin{enumerate}
		\item $[x_1,x_2,x_3][x_4,x_5,x_6]=0$,
		\item $[x_1,x_2]\ldots[x_{2m+3},x_{2m+4}]=0$,
	\end{enumerate}
	are polynomial identities for $UT_2(G^{(2m)})$ and $UT_2(G^{(2m+1)})$.
\end{lemma}
\begin{proof}
	We know that $T(G^{(2m)})=T(G^{(2m+1)})$. So, by \cite[Lemma 10]{CdM}, we have $T(UT_2(G^{(2m)}))=T(UT_2(G^{(2m+1)}))$. Thus, we only need to check the statement for $UT_2(G^{(2m)})$. It is well-known that $[x_1,x_2,x_3][x_4,x_5,x_6]$ is a polynomial identity for $UT_2(G)$, hence, so is for $UT_2(G^{(2m)})$ as well.
	
	Now, consider the polynomial $q$ of (2). Since $q$ is multilinear, we only need to check evaluations of $q$ on matrix units multiplied by elements of $G^{(2m)}$. An evaluation will be automatically zero if two or more variables are substituted by a multiple of $e_{12}$. If all variables assume diagonal values, then we obtain zero again, since the diagonal of $UT_2(G^{(2m)})$ is $G^{(2m)}\oplus G^{(2m)}$.
	
	So, assume that $x_i=ge_{12}$, for some $g$, and let $x_j$ be the variable appearing together with $x_i$. So $[x_i,x_j]=g'e_{12}$. Next, the variables that appear before $[x_i,x_j]$ must be evaluated on some multiple of $e_{11}$, and the variables after $[x_i,x_j]$ must be evaluated on a multiple of $e_{22}$; otherwise we certainly obtain zero. So we have that $q=w_1g'w_2e_{12}$, where $w_1g'w_2$ is a product of elements of $G^{(2m)}$, containing at least $m+1$ commutators of elements of $G^{(2m)}$. So $w_1g'w_2=0$, and $q$ is a polynomial identity of $UT_2(G^{(2m)})$.
\end{proof}

Before we proceed, we recall the following classical result:
\begin{theorem}[Theorem 5.2.1(ii) of \cite{drenskybook}]\label{basis_utn}
	Let $\mathbb{K}$ be any infinite field, and $n\in\mathbb{N}$. The relatively free algebra of the variety generated by the identity
	$$
	[x_1,x_2]\cdots[x_{2n-1},x_{2n}]=0
	$$
	has a basis consisting of all products
	$$
	x_1^{a_1}\ldots x_m^{a_m}[x_{i_{11}},x_{i_{21}},\ldots,x_{i_{p_1 {1}}}]\ldots[x_{i_{1r}},x_{i_{2r}},\ldots,x_{i_{p_rr}}],
	$$
	where the number $r$ of participating commutators is $\le n-1$ and the indices in each commutator $[x_{i_{1s}},x_{i_{2s}},\ldots,x_{i_{p_ss}}]$ satisfy $i_{1s}>i_{2s}\le\cdots\le i_{p_ss}$.\qed
\end{theorem}

\begin{remark}
	Now, let $\mathcal{T}_m$ be the T-ideal generated by the identities of Lemma \ref{lem2}. We notice the following fact. Assume we have a multilinear polynomial of the following kind:
	$$
	[x_{i_1},x_{i_2}]\ldots[x_{2r-1},x_{2r}][y_1,\ldots,y_s][x_{j_1},x_{j_2}]\ldots[x_{j_{2t-1}},x_{j_{2t}}],
	$$
	where $s\ge3$. Then, using the identities of Lemma \ref{lem1}, modulo the polynomial $[x_1,x_2,x_3][x_4,x_5,x_6]$, we can order $i_1<\cdots<i_{2r}$, and $j_1<\cdots<j_{2t}$.
\end{remark}

Consider the following family of polynomials:
\begin{align}
\begin{split}\label{pol1}
[x_{i_1},x_{i_2}]\ldots[x_{i_{2r-1}},x_{i_{2r}}][x_{j_1},\ldots,x_{j_s}][x_{k_1},x_{k_2}]\ldots[x_{k_{2t-1}},x_{k_{2t}}],\\%
r\ge0, t\ge0, \quad r+t\leq m,\quad i_1<i_2<\cdots<i_{2r},\quad k_1<k_2<\cdots<k_{2t},\\%
s>2,\quad j_1>j_2<j_3<\cdots<j_s.
\end{split}
\end{align}

Also, for each $t\in\mathbb{N}$, let $\mathfrak{B}^{(t)}_m$ be a basis of
$$
\mathrm{Span}\{[x_{\sigma(1)},x_{\sigma(2)}]\cdots[x_{\sigma(2t-1)},x_{\sigma(2t)}]\mid\sigma\in\mathcal{S}_{2t}\}+T(UT_2(G^{(2m)}))/T(UT_2(G^{(2m)}).
$$
Denote $\mathfrak{B}_m=\bigcup_{t\in\mathbb{N}}\mathfrak{B}_m^{(t)}$. Let $\mathfrak{T}_m$ be the set of polynomial identities of $UT_2(E^{(2m)})$ given by a linear combination of product of commutators of length 2 (note that $\mathfrak{T}_m\ne0$, since it contains identity (d) of Lemma \ref{lem1}).

\begin{lemma}
	The polynomials \eqref{pol1} and $\mathfrak{B}_m$ generate the proper multilinear polynomials in $\mathbb{K}\langle X\rangle$ modulo $\mathcal{T}_m+\mathfrak{T}_m$.
\end{lemma}
\begin{proof}
	Since $[x_1,x_2]\ldots[x_{2m+3},x_{2m+4}]\in\mathcal{T}_m$, it is enough to write the elements of the relatively free algebra $F(UT_{m+2}(\mathbb{K}))$ as a linear combination of polynomials of kind \eqref{pol1}, and elements of $\mathfrak{B}_m$. From Theorem \ref{basis_utn}, and since $[x_1,x_2,x_3][x_4,x_5,x_6]=0$, it is enough to consider a polynomial $q$ of kind
	$$
	q=[x_{i_1},x_{i_2}]\ldots[x_{i_{2r-1}},x_{i_{2r}}][x_{j_1},\ldots,x_{j_s}][x_{k_1},x_{k_2}]\ldots[x_{k_{2t-1}},x_{k_{2t}}].
	$$
	If $s>2$, then from the remark above, we can order $i_1<\cdots<i_{2r}$, and $k_1<\cdots<k_{2t}$, and we are done. If $s=2$, then $q$ is a product of commutators of length 2. So, $q$ is a linear combination of elements $\mathfrak{B}_m$ modulo $\mathfrak{T}_m$, by definition.
\end{proof}

\begin{lemma}
	The family of polynomials given by \eqref{pol1} and $\mathfrak{B}_m$ are linearly independent modulo $T(UT_2(G^{(2m)}))$.
\end{lemma}
\begin{proof}
	Consider a multilinear polynomial identity $f\in T(UT_2(G^{(2m)}))$, and write $f=f_1+f_2$, where $f_1$ is a linear combination of the polynomials \eqref{pol1}, and $f_2$ is a linear combination of polynomials in $\mathfrak{B}_m$. For some $s>2$, consider the following evaluation $\psi$:
	\begin{align*}
	&x_{i_1}=g_1e_{11},\ldots,x_{i_{2r}}=g_{2r}e_{11},\\%
	&x_{j_2}=e_{12},\\%
	&x_{j_1}=x_{j_3}=\cdots=x_{j_s}=e_{11},\\%
	&x_{k_1}=g_{2r+1}e_{22},\ldots,x_{k_{2t}}=g_{2r+2t}e_{22}.
	\end{align*}
	Then, any polynomial which is the product of more than $r+t+1$ commutators annihilate. Note that, since $s>2$, this evaluation gives $\psi(f_2)=0$. Among the polynomials of type \eqref{pol1}, there is a single polynomial having a nonzero evaluation, namely
	$$
	[x_{i_1},x_{i_2}]\ldots[x_{i_{2r-1}},x_{i_{2r}}][x_{j_1},\ldots,x_{j_s}][x_{k_1},x_{k_2}]\ldots[x_{k_{2t-1}},x_{k_{2t}}].
	$$
	This proves that $f_1=0$. So $f=f_2$. By the choice of $\mathfrak{B}_m$, we obtain $f_2=0$ and we are done.
\end{proof}

As a consequence, we have the following.
\begin{theorem}
	For $m\in\mathbb{N}$, set $\mathcal{T}_m=\langle[x_1,x_2,x_3][x_4,x_5,x_6],[x_1,x_2]\ldots[x_{2m+3},x_{2m+4}]\rangle$. Then,
	$$
	T(UT_2(G^{(2m)}))
	=\mathcal{T}_m+\mathfrak{T}_m,
	$$
	where $\mathfrak{T}_m$ is the set of polynomial identities given by linear combination of product of commutators of length 2.\qed
\end{theorem}

\subsection{The case $UT_2(G^{(2m)})$, for $1\le m\le 2$} For small $k$, we can prove that we do not need the set $\mathfrak{T}_m$ of the previous result.

\begin{lemma}
	Consider the following family of polynomials:
	\begin{align}
	\begin{split}
	\label{pol3}
	[x_{i_1},x_{i_2}]\cdots[x_{i_{2t-1}},x_{i_{2t}}],\\%
	i_1>i_2,\quad i_3>i_4,\ldots,\quad i_{2t-1}>i_{2t}.
	\end{split}
	\end{align}
	Fix any $m\in\mathbb{N}$, and let $t\le\min\{3,m\}$. Then, the polynomials \eqref{pol3} of degree $2t$ generate the subspace spanned by all (multilinear) product of commutators of length 2 of degree $2t$ of $\mathbb{K}\langle X\rangle$ modulo $\mathcal{T}_m$, and they are linearly independent modulo $T(UT_2(G^{(2m)}))$.
\end{lemma}
\begin{proof}
	The assertion that these polynomials generate all multilinear product of commutators of length 2 modulo $\mathcal{T}_m$ is direct from Theorem \ref{basis_utn}. So we only need to prove the linearly independence part. If $t=1$, then there is nothing to do.
	
	So, let $f$ be a linear combination of polynomials \eqref{pol3}, $\deg f=4$. The evaluation
	\begin{align*}
	x_{i_1}&=g_1e_{11},\\%
	x_{i_2}&=g_2e_{11},\\%
	x_{i_3}&=e_{12},\\%
	x_{i_4}&=e_{22},
	\end{align*}
	will make all product of commutators zero, but $[x_{i_1},x_{i_2}][x_{i_3},x_{i_4}]$. Thus, the elements of degree 4 are linearly independent.
	
	Similarly, if $\deg f=6$, then the evaluation
	\begin{align*}
	x_{i_1}&=g_1e_{11},\\%
	x_{i_2}&=g_2e_{11},\\%
	x_{i_3}&=e_{12},\\%
	x_{i_4}&=e_{22},\\%
	x_{i_5}&=g_3e_{22},\\%
	x_{i_6}&=g_4e_{22},
	\end{align*}
	will make all product of commutators zero, but $[x_{i_1},x_{i_2}][x_{i_3},x_{i_4}][x_{i_5},x_{i_6}]$. This concludes the proof.
\end{proof}

As a consequence, if $1\le m\le 2$, then $\mathfrak{T}_m\subseteq\langle[x_1,x_2]\ldots[x_{2m-1},x_{2m}]\rangle$. Thus, using the lemmas from the previous section, we see that
$$
\mathcal{T}_m\subseteq  T(UT_2(G^{(2m)}))\subseteq\mathcal{T}_m.
$$
We proved:
\begin{theorem}
	For $1\le m\le2$, set $$
	\mathcal{T}_m=\langle[x_1,x_2,x_3][x_4,x_5,x_6],[x_1,x_2]\ldots[x_{2m+3},x_{2m+4}]\rangle.
	$$
	Then, $T(UT_2(G^{(2m)}))=T(UT_2(G^{(2m+1)}))=\mathcal{T}_m.$
	\qed
\end{theorem}

\section{Polynomial identities for $F_k(UT_2(G))$\label{sec5}}
Let us denote by $T_1$ the T-ideal generated by $[x_1,x_2,x_3]$ and by $T_2$ the T-ideal generated by $[x_1,x_2,x_3][x_4,x_5,x_6]$.

\begin{lemma}\label{reordering} Let $n\geq 1$. If $f$ is defined as
	\[f=u_0[v_1,v_2,v_3]u_1[w_{\sigma(1)},w_{\sigma(2)}]u_2\cdots
	u_{n}[w_{\sigma(2n-1)},w_{\sigma(2n)}]u_{n+1},\] with $u_i,v_i,w_i\in K\langle X \rangle$ for all $i$ and
	$\sigma \in S_{2n}$, then
	\[f = (-1)^{\sigma}u_0[v_1,v_2,v_3][w_1,w_2]\cdots [w_{2n-1},w_{2n}]u_1 \cdots u_{n+1} \mod T_2.\]
\end{lemma}

\begin{proof}
	After using the identity
	\[c[a,b]=[a,b]c-[a,b,c],\]
	$n$ times, we obtain
	\begin{equation}\label{igualda}
	f = u_0[v_1,v_2,v_3][w_{\sigma(1)},w_{\sigma(2)}]\cdots [w_{\sigma(2n-1)},w_{\sigma(2n)}]u_1\cdots u_{n+1} \mod T_2.
	\end{equation}
	Since
	\[[x_1,x_2][x_3,x_4] = -[x_1,x_3][x_2,x_4] \mod T_1,\]
	the identity
	\[[w_{\sigma(1)},w_{\sigma(2)}]\cdots
	[w_{\sigma(2n-1)},w_{\sigma(2n)}] = (-1)^{\sigma}[w_1,w_2]\cdots [w_{2n-1},w_{2n}] \mod T_1
	\] holds.
	The above identity and (\ref{igualda}) imply
	\[f =  (-1)^{\sigma}u_0[v_1,v_2,v_3][w_1,w_2]\cdots [w_{2n-1},w_{2n}]u_1 \cdots u_{n+1} \mod T_2.\]
\end{proof}

\begin{remark}
	An analogous version of the above lemma, is also true, if one considers the factor $[v_1,v_2,v_3]$ at the end of the monomial. The proof is completely analogous.
\end{remark}

\begin{lemma}\label{idF2}
	If $m\geq 1$, the following polynomials are identities for $F_{k}(UT_2(G))$, for $k\leq 2m+1$:
	\begin{enumerate}
		\item $[x_1,x_2,x_3][x_4,x_5,x_6]$;
		\item $[x_1,x_2,x_3][x_4,x_5]\cdots [x_{2m+4},x_{2m+5}]$;
		\item $[x_1,x_2]\cdots [x_{2m+1},x_{2m+2}][x_{2m+3},x_{2m+4},x_{2m+5}]$;
		\item $[x_1,x_2]\cdots [x_{4m+3},x_{4m+4}]$.
	\end{enumerate}
\end{lemma}

\begin{proof}
	First we observe that it is enough to prove the result for $k=2m+1$. We use the model for the relatively free algebra of rank $2m+1$ of $UT_2(G)$ given in section 2, i.e., the subalgebra $\mathcal{U}_{2m+1}$ of $UT_2(F(G))$ generated by the generic matrices $\xi_1,\dots,\xi_{2m+1}$, where
	
	\[\xi_i=\begin{pmatrix}
	x_{11}^{(1)}+T(G) & x_{12}^{(i)}+T(G)\\
	0                 & x_{22}^{(2)}+T(G)
	\end{pmatrix}\]
	We observe that the set $A_{1,1}=\{p\,|\, p \text{ is the entry (1,1) of some element of } \mathcal{U}_{2m+1}\}$ is an algebra, isomorphic to $F_{2m+1}(UT_2(G))$ of $G$, in the variables $x_{11}^{(1)}, \dots, x_{11}^{(2m+1)}$.
	Analogously, the set $A_{2,2}=\{p\,|\, p \text{ is the entry (2,2) of some element of } \mathcal{U}_{2m+1}\}$ is an algebra, isomorphic to the relatively free algebra of rank $2m+1$ of $G$, in the variables $x_{22}^{(1)}, \dots, x_{22}^{(2m+1)}$. In particular, they satisfy the polynomial identities 
	\[[x_1,x_2,x_3]  \text{ and } [x_1,x_2]\cdots [x_{2m+1},x_{2m+2}].\]
	It is clear that (1) is a polynomial identity, since it is an identity for $UT_2(G)$.
	
	To show that (2) and (3) are identities, it is enough to verify they vanish under substitution of variables by monomials in the variables $\xi_i$. By Lemma \ref{reordering} and using the identity $[ab,c]=a[b,c]+[a,c]b$, it is enough to show that they vanish under substitution of variables by the generic elements $\xi_i$, $i\in {1,\dots, 2m+1}$.
	
	Now one verifies that the substitution of such elements in the polynomials $[x_1,x_2,x_3]$ and $[x_4,x_5]\cdots [x_{2m+4},x_{2m+5}]$ yields matrices which are multiple of the unit matrix $e_{12}$ by an element of $F(G)$, since these polynomials are identities for $A_{1,1}$ and for $A_{2,2}$. As a consequence, the product of the evaluations of such polynomials in both orders vanishes, showing that (2) and (3) are identities for $F_{2m+1}(UT_2(G))$.
	
	Again, to prove that (4) is an identity, it is enough to verify it vanishes under substitution of variables by monomials in the variables $\xi_i$. After using several times the identity $[ab,c]=a[b,c]+[a,c]b$, one obtains a linear combination of elements of the form
	\[u_0[y_1,y_2]u_1[y_3,y_4]u_2\cdots u_{2m+1}[y_{4m+3},y_{4m+4}]u_{2m+2},\]
	where the $u_i$ are elements of $\mathcal{U}_{2m+1}$ and the $y_i$ are generic matrices $\xi_j$. 
	
	If $0<i\leq 2m+1$, then, by using the identity $c[a,b]=[a,b]c-[a,b,c]$ in the factor $u_i[y_{2i+1},y_{2i+2}]$, it turns into $[y_{2i+1},y_{2i+2}]u_i-[y_{2i+1},y_{2i+2},u_i]$. Now, since $i\leq 2m+1$ one observes that using Lemma \ref{reordering} and the fact that (2) is an identity, the component of the sum corresponding to the triple commutator vanishes.
	Applying such procedure several times, we obtain that the elements $u_1,\dots,u_{2m+1}$ can be moved to the middle of the monomial (just after the $(m+1)$-th commutator).
	In a analogous way, using the remark after Lemma \ref{reordering} and the fact that (3) is an identity, we obtain that if $m+1<i\leq 2m+1$, the elements $u_i$ can also be moved to the middle of the monomial, i.e.,

	\[u_0[y_1,y_2]u_1[y_3,y_4]u_2\cdots u_{2m+1}[y_{4m+3},y_{4m+4}]u_{2m+2}=\]
	\[=u_0 [y_1,y_2]\cdots[y_{2m+1},y_{2m+2}]u_1\cdots u_{2m+1}[y_{2m+3},y_{2m+4}]\cdots[y_{4m+3},y_{4m+4}]u_{m+2}\]
	
	Since the product of $2m+1$ commutators is a multiple of $e_{12}$, and the above is a product of two multiples of $e_{12}$, we obtain that the above element is zero in $\mathcal{U}_{2m+1}$.
\end{proof}

\section{Conclusion}
Describing the ideal of identities of the relatively free algebras of finite rank of a given variety may be a very difficult problem. Even the simple case of $UT_2(G)$ is still open even though it seems to be possible to prove it with the canonical techniques.

The role played by the relatively free algebras of finite rank in the description of the subvarieties of a given variety (at least up to asymptotic equivalence) must be studied.

An interesting problem is to consider $\mathfrak V$ a variety of algebras generated by $G(B)$, where $B=B_0+B_1$ is a finite dimensional superalgebra and to investigate if, given an $n$, there exists an $m$ such that $T(F_n(G(B)))=T(G^{(m)}(B))$ (since we have verified that in the variety generated by $UT_2(G)$, $n=m$ does not hold, as in the previously known cases).

In order to know if such questions are true in some generality, it is necessary first to study it for some simpler examples.

For varieties that we know the structure of the $S_n$-module $P_n(\mathcal V)$ (or $\Gamma_n(\mathcal V)$), of multilinear (or proper multilinear) polynomials modulo the identities of $\mathfrak V$, this may be approached by verifying which of the generators of such modules are identities of $F_m(\mathfrak V)$. We will investigate this problem in future projects.

\section{Funding}
T. C. de Mello was supported by grants
\#2018/15627-2 and \#2018/23690-6, S\~ao Paulo Research Foundation (FAPESP), Brazil.

F. Y. Yasumura was suported in part by the Coordena\c c\~ao de Aperfei\c coamento de Pessoal de N\'ivel Superior - Brasil (CAPES) - Finance Code 001.


\bibliographystyle{abbrv}
\bibliography{referencias}

\end{document}